\documentclass[11pt]{amsart}

\newtheorem{thm}{Theorem}[section]
 \newtheorem{cor}[thm]{Corollary}
 \newtheorem{lem}[thm]{Lemma}
 \newtheorem{prop}[thm]{Proposition}
 \theoremstyle{definition}
 \newtheorem{defn}[thm]{Definition}
 
 \theoremstyle{remark}
 \newtheorem{rem}[thm]{Remark}
 \theoremstyle{notation}
 
 \numberwithin{equation}{section}

\newcommand{\eps}{\epsilon}

\newcommand{\beq}{\begin{equation}}
\newcommand{\eeq}{\end{equation}}
\newcommand{\imn}{\mathbf{i}}
\newcommand{\gdim}{\mathbf{n}}

  \newcommand{\s}{{S}}
 \newcommand{\A}{\mathcal{A}}
 
\newcommand{\f}{\mathcal{F}}
\newcommand{\I}{\mathcal{I}}

  \newcommand{\h}{\mathfrak{h}}
 \newcommand{\bneg}{\mathfrak{b}}

\newcommand{\nneg}{\mathfrak{n}}

 \newcommand{\g}{\mathfrak{g}}

 \newcommand{\Ad}{\textrm{Ad}}
 \newcommand{\ad}{\textrm{ad}}
  
 \newcommand{\Z}{\mathbb{Z}}

 \newcommand{\Lie}{\mathfrak{L}}
 \newcommand{\q}{z}
 
\newcommand{\gauge}{\mathcal{N}}

 \newcommand{\lop}[1]{\mathfrak{L}(#1)}
\newcommand{\bil}[2]{{\langle #1 | #2\rangle}}

\newcommand{\npos}{\mathfrak{n}^+}

\begin{document}

\title[ Frobenius manifolds and  $W$-algebras]
 {Frobenius manifolds from principal classical $W$-algebras
 }

\author{Yassir Ibrahim Dinar }

\address{Faculty of Mathematical Sciences, University of Khartoum,  Sudan; and The Abdus Salam International Centre for Theoretical Physics (ICTP), Italy. Email: dinar@ictp.it.}
\email{dinar@ictp.it}


\begin{abstract}
We obtain  polynomial Frobenius manifolds from classical $W$-algebras associated to principal nilpotent elements in simple Lie algebras.
\end{abstract}
\maketitle
\tableofcontents
\section{Introduction}

This work is a continuation of \cite{mypaper} where we began to develop a construction  of algebraic Frobenius manifolds from Drinfeld-Sokolov reduction to support a Dubrovin conjecture.

A Frobenius manifold is a manifold $M$ with the structure of Frobenius algebra on the tangent space $T_t$ at any point
$t \in  M $ with certain compatibility conditions \cite{DuRev}. We say $M$ is semisimple or massive if $T_t$ is semisimple for generic $t$. This structure locally
corresponds to a potential  satisfying  a system of partial differential equations  known in topological field theory as the  Witten-Dijkgraaf-Verlinde-Verlinde (WDVV) equations. We say $M$ is algebraic if, in the flat coordinates, the potential is an algebraic function.  Dubrovin conjecture is stated as follows: Semisimple irreducible algebraic Frobenius manifolds
with positive degrees  correspond to quasi-Coxeter (primitive) conjugacy classes in finite Coxeter groups. We discussed in \cite{mypaper} how the examples of algebraic Frobenius manifolds constructed from Drinfeld-Sokolov reduction support this conjecture.

Let $e$ be a \textbf{principal nilpotent element} in a simple Lie algebra $\g$ over $\mathbb{C}$. We fix, by using the Jacobson-Morozov theorem,   a semisimple element $h$ and a nilpotent element $f$ such that $\A=\{e,h,f\}$ is an $sl_2$-triple. Let $\kappa+1$ be the Coxeter number of $\g$. We prove the following

\begin{thm}\label{main thm}
  The \textbf{Slodowy slice}
\beq
Q':=e+\ker \ad ~f
\eeq
has a natural structure of polynomial Frobenius manifold of  degree $\kappa-1\over \kappa+1$.
\end{thm}

Let us recall  some structures related to the principal nilpotent element $e$. The element  $h\in \A$ defines a $\mathbb{Z}$-grading on $\g$ called the Dynkin grading given as follows
\begin{equation}
\g=\oplus_{i\in \Z} \g_i, ~~\g_i=\{q\in \g: ad~h(q)=i q\}.
\end{equation}
We  fix   below a certain  nonzero element $a\in \g_{-2\kappa}$. It will follow  from the work of Kostant \cite{kostBetti}  that  $y_1=e+a$ is regular semisimple. The Cartan subalgebra $\h'=\ker \ad~ y_1$ is called the  opposite Cartan subalgebra.

Our main idea is to use the theory of local bihamiltonian structure  on a loop space to construct the polynomial
Frobenius manifold on $Q'$. Recall that a bihamiltonian structure on a manifold $M$ is two compatible Poisson brackets on $M$. It is well known that the dispersionless limit of a local bihamiltonian structure
on the loop space $\lop M$ of a finite dimensional manifold $M$ (if it exists) always
gives a bihamiltonian structure of hydrodynamic type:
\begin{equation}\{t^i(x),t^j(y)\}_{1,2} = g^{ij}_{1,2}(t(x))
  \delta'(x-y) + \Gamma^{ij}_{1,2;k}(t(x)) t^k_x \delta(x-y),\end{equation}
defined on the loop space $\lop M$. This in turn gives a flat pencil of
metrics $g^{ij}_{1,2}$ on $M$ which under some assumptions
corresponds to a Frobenius structure on $M$  \cite{DFP}.

 We perform  Drinfeld-Sokolov reduction \cite{DS} (see also \cite{mypaper} or \cite{BalFeh1}) using the representation theory of $\A$ and the properties of $\h'$  to obtain a bihamiltonian structure on the affine loop space  \beq Q=e+ \lop {\ker \ad\, f}.\eeq
To this end we start by defining a bihamiltonian structure $P_1$ and $P_2$ in $\lop \g$. The Poisson structure $P_2$ is the standard Lie-Poisson structure and $P_1$ depends on the adjoint action of $a$.  In the Drinfeld-Sokolov  reduction the space $Q$ will be transversal to an action of the adjoint group of $\lop {\nneg}$ on a suitable affine subspace of $\lop\g$. Here $\nneg$ is the subalgebra
\beq\nneg:=\bigoplus_{i\leq -2}\g_i\eeq
 The space of local functionals with densities in the ring $R$ of invariant differential polynomials of this action is closed under $P_1$ and $P_2$. This defines the Drinfeld-Sokolov bihamiltonian structure on $Q$ since the coordinates of $Q$ can be interpreted as generators of the ring $R$. The second reduced Poisson structure on $Q$ is called the \textbf{classical $W$-algebra}. We call it \textbf{principal} since it is related to the principal nilpotent element. We then prove that the Drinfeld-Sokolov bihamiltonian structure admits a dispersionless limit and gives the promised polynomial Frobenius manifold.

 We mention  that from  the work of Dubrovin \cite{DCG} and Hetrling \cite{HER} semisimple polynomial Frobenius manifolds with positive degrees are already classified. They correspond to Coxeter conjugacy classes in Coxeter groups. Dubrovin constructed all these polynomial Frobenius manifolds on the orbit spaces of Coxeter groups using the results of \cite{saito}. There is another method to obtain the classical $W$-algebra  associated to principal nilpotent elements known in the literature as Muira type transformation \cite{DS}.  It was used in \cite{DubCentral} (see also \cite{dinarphd})  to prove that the dispersionless limit of the Drinfeld-Sokolov bihamiltonian structure gives  the  polynomial Frobenius manifold defined  on the orbit space of the corresponding Weyl group \cite{DCG}. The proof depends also on the invariant theory of  Coxeter groups. In the present work we give a new method  to uniform the construction of  polynomial Frobenius manifolds from  Drinfeld-Sokolov reduction   which  depends only on the theory of opposite Cartan subalgebras.

\section{Preliminaries}

\subsection{Frobenius manifolds and  local bihamiltonian structures}

Starting we want to recall some definitions and review the construction of Frobenius manifolds from local bihamiltonian structure of hydrodynamics type.

A \textbf{Frobenius manifold} is a manifold $M$ with
the structure of Frobenius algebra on the tangent space $T_t$ at
any point $t \in M $ with certain compatibility conditions \cite{DuRev}.  This structure  locally corresponds to a potential
$\mathbb{F}(t^1,...,t^n)$ satisfying the WDVV equations
\begin{equation} \label{frob}
 \partial_{t^i}
\partial_{t^j}
\partial_{t^k} \mathbb{F}(t)~ \eta^{kp} ~\partial_{t^p}
\partial_{t^q}
\partial_{t^r} \mathbb{F}(t) = \partial_{t^r}
\partial_{t^j}
\partial_{t^k} \mathbb{F}(t) ~\eta^{kp}~\partial_{t^p}
\partial_{t^q}
\partial_{t^i} \mathbb{F}(t)
  \end{equation}
where $(\eta^{-1})_{ij} = \partial_{t^n} \partial_{t^i} \partial_{t^j}
\mathbb{F}(t)$ is a constant matrix. Here we assume that the quasihomogeneity condition takes the form
\begin{equation}
\sum_{i=1}^n d_i t_i \partial_{t^i} \mathbb{F}(t) = \left(3-d \right) \mathbb{F}(t)
\end{equation}
where $d_n=1$. This condition defines {\bf the degrees} $d_i$ and {\bf the charge} $d$ of  the Frobenius structure on  $M$. If $\mathbb{F}(t)$ is an algebraic function we call $M$ an
\textbf{algebraic Frobenius manifold}.

 Let $\lop M$ denote the loop space of $M$, i.e the space of smooth maps from the circle to $M$. A local Poisson bracket $\{.,.\}_1$ on $\lop M$ can be written in the form \cite{DZ}
\begin{equation} \label{genLocPoissBra}\{u^i(x),u^j(y)\}_1=
\sum_{k=-1}^\infty \epsilon^k \{u^i(x),u^j(y)\}^{[k]}_1.
 \end{equation}
 Here $\epsilon$ is just a parameter and
 \begin{equation}\label{genLocBraGen}
\{u^i(x),u^j(y)\}^{[k]}_1=\sum_{s=0}^{k+1} A_{k,s}^{i,j}
\delta^{(k-s+1)}(x-y),
 \end{equation}
where  $A_{k,s}^{i,j}$ are homogenous polynomials in $\partial_x^j
u^i(x)$ of degree $s$ (we assign degree $j$ to
$\partial_x^j u^i(x)$)and $\delta(x-y)$ is the Dirac  delta function defined by
\[\int_{S^1} f(y) \delta(x-y) dy=f(x).\]
 The first terms can be written
as follows
\begin{eqnarray}
  \{u^i(x),u^j(y)\}^{[-1]}_1 &=& F^{ij}_1(u(x))\delta(x-y) \\
  \{u^i(x),u^j(y)\}^{[0]}_1 &=& g^{ij}_{1}(u(x)) \delta' (x-y)+ \Gamma_{1k}^{ij}(u(x)) u_x^k \delta (x-y)
\end{eqnarray}
 Here  the entries $g^{ij}_1(u)$, $F^{ij}_1(u)$ and $\Gamma_{1k}^{ij}(u)$ are smooth functions on the finite dimension space $M$. We note that, under the change of coordinates on $M$ the matrices $g^{ij}_1(u)$, $F^{ij}_1(u)$ change as a $(2,0)$-tensors.

The matrix $F^{ij}_1(u)$  defines a Poisson structure on
$M$. If $F^{ij}_1(u(x))=0$ and   $\{u^i(x),u^j(y)\}^{[0]}_1\neq 0$ we say the Poisson bracket admits a \textbf{dispersionless limit}. If the Poisson bracket admits a dispersionless limit then  $\{u^i(x),u^j(y)\}^{[0]}_1$ defines a
 Poisson bracket on $\lop M$ known as \textbf{Poisson bracket of
hydrodynamic type}. By nondegenerate Poisson bracket of
hydrodynamic type we mean those with the metric $g^{ij}_1$ is nondegenerate. In
this case the matrix $g^{ij}_1(u)$ defines a contravariant flat metric on the cotangent space $T^*M$ and
$\Gamma_{1k}^{ij}(u)$ is its  contravariant Levi-Civita connection \cite{DN}.

Assume there are two Poisson structures $\{.,.\}_2$ and  $\{.,.\}_1$ on $\lop M$ which form  a bihamiltonian structure, i.e $\{.,.\}_\lambda:= \{.,.\}_2+\lambda \{.,.\}_1$ is a Poisson structure on $\lop M$ for every $\lambda$. Consider the notations for the leading terms of $\{.,.\}_1$ given above and write the leading terms of $\{.,.\}_2$ in the form
\begin{eqnarray}
  \{u^i(x),u^j(y)\}^{[-1]}_2 &=& F^{ij}_2(u(x))\delta(x-y) \\
  \{u^i(x),u^j(y)\}^{[0]}_2 &=& g^{ij}_{2}(u(x)) \delta' (x-y)+ \Gamma_{2k}^{ij}(u(x)) u_x^k \delta (x-y)
\end{eqnarray}
Suppose that $\{.,.\}_1$ and $\{.,.\}_2$ admit a dispersionless limit as well as $\{.,.\}_\lambda$ for generic $\lambda$. In addition, assume the corresponding Poisson brackets of hydrodynamics type are nondegenerate. Then by definition $g_1^{ij}(u)$ and $g_2^{ij}(u)$ form what is called  \textbf{flat pencil of metrics} \cite{DFP}, i.e $g_\lambda^{ij}(u):=g_2^{ij}(u)+\lambda g_1^{ij}(u)$ defines a flat metric on $T^*M$ for generic $\lambda$ and its  Levi-Civita connection is given by $\Gamma_{\lambda k}^{ij}(u)=\Gamma_{2k}^{ij}(u)+\lambda \Gamma_{1k}^{ij}(u)$.

\begin{defn} \label{def reg} A contravariant flat pencil of metrics on a manifold $M$ defined by  the matrices $g_1^{ij}$ and $g_2^{ij}$ is
called \textbf{ quasihomogenous of  degree} $d$ if there exists a
function $\tau$ on $M$ such that the vector fields
\begin{eqnarray}  E&:=& \nabla_2 \tau, ~~E^i
=g_2^{is}\partial_s\tau
\\\nonumber  e&:=&\nabla_1 \tau, ~~e^i
= g_1^{is}\partial_s\tau  \end{eqnarray} satisfy the following
properties
\begin{enumerate}
\item $ [e,E]=e$.

\item $ \Lie_E (~,~)_2 =(d-1) (~,~)_2 $.
\item $ \Lie_e (~,~)_2 =
(~,~)_1 $.
\item
$ \Lie_e(~,~)_1
=0$.
\end{enumerate}
Here for example $\Lie_E$ denote the Lie derivative  along the vector field $E$ and $(~,~)_1$ denote the metric defined by the matrix  $g^{ij}_1$. In addition, the  quasihomogenous flat pencil of metrics is called \textbf{regular} if  the
(1,1)-tensor
\begin{equation}\label{regcond}
  R_i^j = {d-1\over 2}\delta_i^j + {\nabla_1}_i
E^j
\end{equation}
is  nondegenerate on $M$.
\end{defn}

The connection between the theory of Frobenius manifolds and flat pencil of metrics is encoded in the following theorem

\begin{thm}\cite{DFP}\label{dub flat pencil}
A contravariant quasihomogenous regular  flat pencil of metrics of degree $d$ on a manifold $M$ defines a Frobenius structure on $M$ of the same degree.
 \end{thm}

It is well known that from a Frobenius manifold we always have a flat pencil of metrics but it does not necessary satisfy the regularity condition \eqref{regcond}. In the notations of  \eqref{frob} from a Frobenius structure on $M$, the flat pencil of metrics is
found from the relations \begin{eqnarray}\label{frob eqs} \eta^{ij}&=&g_1^{ij} \\
\nonumber g_2^{ij}&=&(d-1+d_i+d_j)\eta^{i\alpha}\eta^{j\beta}
\partial_\alpha
\partial_\beta \mathbb{F}
\end{eqnarray}
This flat pencil of metric is quasihomogenous  of degree $d$ with $\tau =t^1$. Furthermore we have
\begin{equation}
E=\sum_i d_i t^i {\partial{t^i}},~~~e={\partial_{t^n}}
\end{equation}

\subsection{Principal nilpotent element and opposite Cartan subalgebra}
We  review some facts about  principal nilpotent elements in  simple Lie algebra we need to perform the Drinfeld-Sokolov reduction. In particular, we recall   the  concept of the opposite Cartan subalgebra introduced by Kostant which is the  main  ingredient in this work.

Let $\g$ be a simple Lie algebra over $\mathbb C$ of rank $r$.  We fix a principal nilpotent element $e\in \g$. By definition a nilpotent element is called principal if $\g^e:=\ker \ad e$ has dimension equals to $r$.  Using  the Jacobson-Morozov theorem we fix a semisimple element $h$ and a nilpotent element $f$ in $\g$ such that $\{e,h,f\}$ generate $sl_2$ subalgebra $\A\subset \g$, i.e
\begin{equation}
[h,e]=2 e, ~~~ [h,f]=-2f,~~~[e,f]=h.
\end{equation}

The element  $h$ define a $\mathbb{Z}$-grading on $\g$ called the Dynkin grading given as follows
\begin{equation}
\g=\oplus_{i\in \Z} \g_i, ~~\g_i=\{q\in \g: ad~h(q)=i q\}.
\end{equation}
 It is well known that $\g_i=0$ if $i$ is odd and  \beq \bneg=\oplus_{i\leq 0} \g_i\eeq is a Borel subalgebra with   \beq \nneg=\oplus_{i\leq
{-2}}\g_i=[\bneg,\bneg]\eeq is  a nilpotent subalgebra.

We normalize the invariant bilinear from $\bil . . $ on $\g$ such that $\bil e f=1$ and we denote the exponents of the Lie algebra $\g$ as follows
\begin{equation}
1=\eta_1<\eta_2\leq \eta_3\ldots\leq\eta_{r-1}<\eta_r.
\end{equation}
We will refer to the number $\eta_r$ by $\kappa$. Recall that $\kappa+1$ is the Coxeter number of $\g$ and the exponents satisfy the relation
\beq \eta_i+\eta_{r-i+1}=\kappa+1.\eeq
We also recall that  for all simple Lie algebras the exponents  are different except for the Lie algebra of type $D_{2n}$ the exponent $n-1$ appears twice.

Consider the restriction of the adjoint representation of $\g$  to $\A$. Under this restriction $\g$ decomposes to irreducible $\A$-submodules
\begin{equation}
\g=\oplus V^i.
\end{equation}
with $ \dim V^i=2 \eta_i+1$ \cite{HumLie}. We normalize this decomposition by using the following proposition
 \begin{prop}\label{reg:sl2:normalbasis}
 There exists a decomposition of $\g$ into a sum of irreducible $\A$-submodules $\g=\oplus_{i=1}^{r} V^i$ in such a way that there is a    basis $X_I^i, I=-\eta_i,-\eta_i+1,...,\eta_i$ in each $V^i, ~i=1,\ldots,r$ satisfying the following relations
\begin{equation}\label{sl2expand}
X_{I}^i={1\over (\eta_i+I)!} \ad \,e^{\eta_i+I}~X_{-\eta_i}^i~ ,~~~~I=-\eta_i,-\eta_i+1,\ldots, \eta_i.
\end{equation}
and
\beq\label{sl2bilinear}
<X_I^i,X_J^j>=\delta_{i,j}\delta_{I,-J} (-1)^{\eta_i-I+1}{2\eta_i\choose \eta_i-I}.
\eeq
Furthermore
\begin{eqnarray}\label{sl2relation}
\ad \, h\,X_I^i&=& 2I X_I^i.\\\nonumber
\ad \, e\,X_I^i&=& (\eta_i+I+1) X_{I+1}^i.\\\nonumber
\ad \, f\, X_I^i &=& (\eta_i -I+1) X_{I-1}^i.
\end{eqnarray}
\end{prop}
\begin{proof}
The proof that one could compose the Lie algebra as irreducible $\A$-submodules satisfying \eqref{sl2expand} and \eqref{sl2relation} is standard and can be found  in \cite{HumLie} or \cite{kostBetti}. Let $\g=\oplus_{i=1}^{r} V^i$ be such  decomposition. It is easy to prove $\bil {V^i}{V^j}=0$ in the case  $\eta_i\neq\eta_j$ by applying the step operators $\ad~ e$ and using the  invariance of the bilinear form. Hence the proof is reduced to the case of irreducible $\A$-submodules of the same dimension. But there is at most two irreducible submodules of the same dimension. Assume $V^{i_1}$ and $V^{i_2}$ have  the same dimension and denote the corresponding basis  $X^{i_1}_I$ and $X^{i_1}_J$, respectively. Then one can prove by using the step operator $\ad~ e$ that the subspaces $V^{i_1}$ and $V^{i_2}$  are orthogonal if and only if $\bil {X^{i_1}_0}{X^{i_2}_0}= 0$. But it obvious that  the restriction of the invariant bilinear form to ${X^{i_1}_0}$ and ${X^{i_2}_0}$ is nondegenerate. Hence by applying the Gram-Schmidt  procedure we can assume that $\bil {X^{i_1}_0}{X^{i_2}_0}= 0$.
Therefore, we can assume that the given  decomposition satisfying $\bil {V^i}{V^j}=0$ if $i\neq j$. It remains to obtain the normalization \eqref{sl2bilinear}. From the invariance of the bilinear form we have
\begin{equation}
\bil{h.X_{I}^i}{X_J^i}=(2 I) \bil{X_I^i}{X_J^i}
\end{equation}
while
\begin{equation}
-\bil{X_I^i}{h.X_J^i}=-(2 J)\bil{X_I^i}{X_J^i}
\end{equation}
Therefore $\bil{X_I^i}{X_J^j}=0$ if $I+J\neq 0$. We calculate using the step operator $\ad ~e$ where  $I\geq 0$ the value
\begin{eqnarray}
\bil{X_I^i}{X_{-I}^i}&=&{1\over (\eta_i-I) } \bil{X_{I}^i}{e.X_{-I-1}^i} \\\nonumber
&=&{-1\over \eta_i-I } \bil{e.X_{I}^i}{X_{-I-1}^i} \\\nonumber
&=&{(-1)(\eta_i-I+1)\over \eta_i-I } \bil{X_{I+1}^i}{X_{-I-1}^i} \\\nonumber
&=&{(-1)^{\eta_i-I}(\eta_i-I+1)(\eta_i-I+2)\ldots 2\eta_i\over (\eta_i-I)(\eta_i-I-1)\ldots(1) } \bil{X_{\eta_i}^i}{X_{-\eta_i}^i}\\\nonumber
&=& (-1)^{\eta_i-I} {2\eta_i\choose \eta_i-I} \bil{X_{\eta_i}^i}{X_{-\eta_i}^i}.
\end{eqnarray}
The result follows by multiplying $X_I^i$ by the value of  $-\bil{X_{\eta_i}^i}{X_{-\eta_i}^i}^{-1}$. We note that the formula \eqref{sl2bilinear} will give the same result when  replacing  $I$ with $-I$. This ends the proof.

\end{proof}

 Note that the normalized basis for $V^1$ are $X_1^1=-e,~X_0^1=h,~X_{-1}^1=f$ since it is isomorphic to $\A$ as a vector subspace.

It is easy to see that
\beq
\ad~ e: \g_i \to \g_{i+2}
\eeq
is injective for $i\leq -1$ and surjective for $i\geq0$. Hence   the subalgebra $\g^f:=\ker \ad ~f$ has a basis $X_{-\eta_i}^i, ~i=1,\ldots,r$ and
\beq
\bneg=\g^f\oplus \ad ~e (\nneg).
\eeq
The affine space \[Q'=e+\g^f\] is called the \textbf{Slodowy slice}. It is  transversal  to the orbit  of $e$ under the adjoint group action.

We summarize Kostant results about the relation between the   principal nilpotent element $e$ and Coxeter conjugacy class in Weyl group of $\g$.
\begin{thm}\cite{kostBetti}
The element $y_1=e+ X_{-2\kappa}^r$ is regular semisimple. Denote $\h'$ the  Cartan subalgebra containing $y_1$, i.e $\h':=\ker \ad~y_1$ and  consider the adjoint group element $w$ defined by  $w:=\exp {\pi \imn\over \kappa+1}\ad~h$. Then $w$ acts on $\h'$ as a representative of the Coxeter conjugacy class  in the Weyl group acting on $\h'$. Furthermore, the element $y_1$ can be completed to a basis  $y_i,~i=1,\ldots,r$ for $\h'$
having the form
\[ y_i=v_i+u_i, ~~u_i\in \g_{2\eta_i},~ v_i\in \g_{2\eta_i-2(\kappa+1)}\]
and such that $y_i$ is an eigenvector of $w$ with eigenvalue $\exp {\pi \imn\eta_i\over \kappa+1}$.
\end{thm}

Let  $a$ denote the element $X_{-2\kappa}^r$. The  element $y_1=e+a$ is called a \textbf{cyclic element} and the   Cartan subalgebra $\h'=\ker \ad ~y_1$ is called the \textbf{opposite Cartan subalgebra}. We  fix  a basis $y_i$ for $\h'$ satisfying the theorem above.  It is easy to see that  $u_i, i=1,...,r$ form a homogenous  basis for $\g^e$.  We assume the basis $y_i$ are normalized such  that
\begin{equation}
u_i=-X_{\eta_i}^i.
\end{equation}
Form construction this normalization does not effect $y_1$.

Let us define the matrix of the invariant bilinear form on $\h'$
\beq
A_{ij}:=\bil {y_i}{y_j}=-\bil {X_{\eta_i}^i}{v_j}-\bil {v_i}{X_{\eta_j}^i},~i,j=1,\ldots,r.
\eeq
The following proposition summarize some useful properties we need in the following sections.
\begin{prop}\label{Gold}
The matrix $A_{ij}$ is a nondegenerate  and antidiagonal  with respect to the exponents $\eta_i$, i.e $A_{ij}=0, ~{\rm if}~\eta_i+\eta_j\neq\kappa+1$. Moreover, the commutators of $a$ and  $X_{\eta_i}^i$ satisfy the relations
\begin{equation}
{\bil {[a,X_{\eta_i}^i]}{X_{\eta_j-1}^j}\over 2\eta_j }+{\bil {[a,X_{\eta_j}^j]}{X_{\eta_i-1}^i}\over 2 \eta_i}=  A_{ij}
\end{equation}
for all $i,j=1,\ldots,r$.
\end{prop}

\begin{proof}
The matrix $A_{ij}$ is nondegenerate since  the restriction of the invariant bilinear form to a Cartan subalgebra is nondegenerate. The fact that it  is anidiagonal with respect to the exponents follows from the identity
\beq
\bil {y_i}{y_j}=\bil {w y_i}{w y_j}=\exp {(\eta_i+\eta_j)\pi \imn\over \kappa+1}\bil {y_i}{y_j}
\eeq
where  $w:=\exp {\pi \imn\over \kappa+1}\ad~h$. For the second part of the proposition we note that  the commutator  of $y_1= e+a$ and $y_i=v_i-X_{\eta_i}^i$ gives the relation
\begin{equation}
[e,v_i]=[a,X_{\eta_i}^i], ~ i=1,...,r.
\end{equation}
Which  in turn give the following equality for every  $i,~j=1,...,r$
\begin{eqnarray}
\bil {[a,X_{\eta_i}^i]}{X_{\eta_j-1}^j}&=&\bil {[e,v_i]}{X_{\eta_j-1}^j}=-\bil {v_i}{[e,X_{\eta_j-1}^j]}\\\nonumber &=& -2 \eta_j \bil {v_i}{X_{\eta_j}^j}
\end{eqnarray}
but then
\begin{equation}
{\bil {[a,X_{\eta_i}^i]}{X_{\eta_j-1}^j}\over 2\eta_j }+{\bil {[a,X_{\eta_j}^j]}{X_{\eta_i-1}^i}\over 2 \eta_i}= -\bil {v_i}{X_{\eta_j}^j}-\bil {v_j}{X_{\eta_i}^i}= A_{ij}.
\end{equation}
\end{proof}

\section{Drinfeld-Sokolov reduction}

We will review the standard Drinfeld-Sokolov reduction associated with the principal nilpotent element \cite{DS} (see also \cite{mypaper}).

We introduce the following bilinear form on the loop algebra $\lop\g$:
\begin{equation} (u|v)=\int_{S^1}\bil {u(x)}{v(x)} dx,~ u,v \in \lop M,
\end{equation}
and we identify $\lop\g$ with $\lop \g^*$ by means of this bilinear form.  For a functional $\f$ on $\lop\g$ we
define the gradient $\delta \f (q)$ to be the unique element in
$\lop\g$ such that
\begin{equation}
\frac{d}{d\theta}\f(q+\theta
\dot{s})\mid_{\theta=0}=\int_{S^1}\langle\delta \f|\dot{s}\rangle dx
~~~\textrm{for all } \dot{s}\in \lop\g.
\end{equation}

Recall that we fixed an element $a\in \g$ such that $y_1=e+a$ is a cyclic element. Let us introduce a bihamiltonian structure on $\lop \g$ by means of Poisson tensors
\begin{eqnarray}\label{bih:stru on g}
P_2(v)(q(x))&=&{1\over \eps}[\eps \partial_x+q(x),v(x)].\\\nonumber
P_1(v)(q(x))&=&{1\over \eps} [a,v(x)].
\end{eqnarray}
It is well known fact that these define a bihamiltonian structure on $\lop \g$ \cite{MRbook}.

We consider the gauge transformation of the adjoint group $G$ of $\lop\g$  given by
\begin{eqnarray}
q(x)&\rightarrow& \exp \ad~ s(x)( \partial_x+q(x))-\partial_x
\end{eqnarray}
where $s(x),~q(x)\in \lop \g$.  Following Drinfeld and Sokolov \cite{DS}, we consider the restriction of this action  to the adjoint group $\gauge$ of $\lop {\nneg}$.
\begin{prop}(\cite{mypaper}, \cite{Pedroni2})\label{DS as momentum }
The action of $\gauge$ on $\lop\g$ with Poisson tensor \beq P_\lambda:=P_2+\lambda P_1\eeq is
Hamiltonian for all $\lambda$. It admits a momentum map $J$ to be
the projection
\[J:\lop\g\to\lop \npos\]
 where  $\npos$ is the image of $\nneg$ under the killing map. Moreover, $J$ is $\Ad^*$-equivariant.
\end{prop}

We take $e$ as regular value of $J$. Then
\beq
S:=J^{-1}(e)=\lop\bneg+e,
\eeq
 since $\bneg$ is the orthogonal complement to $\nneg$. It follows from the Dynking grading that  the isotropy group  of $e$ is $\gauge$.

 Recall that the space $Q$ is defined as
\beq
Q:=e+\lop {\g^f}.
\eeq
The following proposition  identified $S/\gauge$  with the space $Q$. Which  allows us to  define the set $\mathcal{R} $ of functionals
on $Q$  as functionals on $\s$
which have densities in the ring $R$.
\begin{prop}\cite{DS}
The space $Q$
is a cross section for the action of $\gauge$ on $\s$, i.e for any element $q(x)+e\in \s$ there is a unique element $s(x) \in
\lop \nneg$ such that  \begin{equation}\label{gauge fix} z(x)+e=(\exp \ad~s(x))( \partial_x+q(x))-\partial_x \in Q.\end{equation} The entries of $z(x)$ are generators of the ring
$R$ of differential polynomials on $S$ invariant under the action of  $\gauge$.
\end{prop}

The  Poisson pencil $P_\lambda$ \eqref{bih:stru on g} is reduced on
$Q$ using the following lemma.
\begin{lem}\label{ringInv}\cite{DS}
Let $\mathcal{R}$ be the functionals on $Q$ with densities belongs to $R$. Then $\mathcal{R}$ is a closed subalgebra with respect to the Poisson
pencil $P_\lambda$.
\end{lem}

Hence  $Q$ has a bihamiltonian structure $P_1^Q$ and $P_2^Q$ from $P_1$ and $P_2$, respectively. The reduced Poisson  structure $P_2^Q$ is called a \textbf{classical $W$-algebra}. For a formal definition of classical $W$-algebras see \cite{fehercomp}. We obtain  the reduced bihamiltonian structure by using lemma \ref{ringInv} as follows. We write the coordinates of $Q$ as differential polynomials in the coordinates of $S$ by means of  equation \eqref{gauge fix} and then apply the Leibnitz rule. For $u,v \in R$ the Leibnitz rule have the following form
\begin{equation}
\{u(x),v(y)\}_\lambda={\partial u(x)\over \partial (q_i^I)^{(m)}}\partial_x^m\Big({\partial v(y)\over \partial (q_j^J)^{(n)}} \partial_y^n\big(\{q_i^I(x),q_j^J(y)\}_\lambda\big)\Big)
\end{equation}

The generators of the invariant ring $R$ will have nice properties when we use  the  normalized basis we developed in last section. Let  us begin by writing  the equation of  gauge fixing \eqref{gauge fix} after introducing  a parameter $\tau$ as follows
\begin{eqnarray*}
q(x)+ e&=& \tau \sum_{i=1}^{r} \sum_{I=0}^{\eta_i} q_{i}^I X_{-I}^{i}+e\in \s\\
z(x)+e &=&\tau\sum_{i=1}^{r} z^i(x) X_{-\eta_i}^i+e\in Q\\
s(x)&=&\tau\sum_{i=1}^{r}\sum_{I=1}^{\eta_i} s_{i}^{I}(x) X_{-I}^{i} \in \lop\nneg.
\end{eqnarray*}
Then equation \eqref{gauge fix} expands to
\begin{equation}\label{gauge fixing data }
\begin{split}
\sum_{i=1}^{r} z^i(x) X_{-\eta_i}^i+&\sum_{i=1}^{r}\sum_{I=1}^{\eta_i}(\eta_i-I+1) s_{i}^{I} X_{-I+1}^{i}=\\&\sum_{i=1}^{r} \sum_{I=0}^{\eta_i} q_{i}^I(x) X_{-I}^{i}-\sum_{i=1}^{r}\sum_{I=1}^{\eta_i} \partial_x s_{i}^{I}(x) X_{-I}^{i}+\mathcal{O}(\tau).
\end{split}
\end{equation}
It  obvious that  any invariant $z^i(x)$ has the form
\begin{eqnarray}\label{leading terms1}
z^i(x)&=&q_i^{\eta_i}-\partial_x s_i^{\eta_i}+\mathcal{O}(\tau)\\\nonumber
&=& q_i^{\eta_i}(x)-\partial_x q_{i}^{\eta_i-1}+\mathcal{O}(\tau).
\end{eqnarray}
That is,  we obtained  the linear term of each invariant $z^i(x)$. Furthermore, since    $\bil {e}{f}=1$ then $z^1(x)$ has the expression
\beq
\begin{split}
z^1(x)= q_1^1(x) -\partial_x s_1^1+& \tau \bil {e}{[s_i^1(x) X_{-1}^i, q_{i}^0 X_0^i]}\\&+{1\over2}\tau \bil{e}{[s_i^1(x) X_{-1}^i,[s_i^I(x) X_{-1}^i,e]]}.
\end{split}
\eeq
Which is simplified by using the identity
\beq
[s_i^1(x) X_{-1}^i,[s_i^I(x) X_{-1}^i,e]]=-[s_i^1(x) X_{-1}^i,q_i^0(x) X_0^i]
\eeq
and
\beq
\bil {e}{[s_i^1(x) X_{-1}^i, q_{i}^0 X_0^i]}=-\bil {[s_i^1(x) X_{-1}^i,e]}{q_i^0(x) X_0^i}=(q_i^0(x))^2\bil{X_0^i}{ X_0^i}
\eeq
with $s^1_1(x)=q_1^0(x)$ to the expression
\begin{equation}
z^1(x)= q_1^1(x) -\partial_x q_1^0(x)+{1\over 2} \tau\sum_i (q_i^0(x))^2\bil{X_0^i}{ X_0^i}
\end{equation}
The invariant $z^1(x)$ is called  a  \textbf{Virasoro density} and the expression  above  agree with \cite{BalFeh}.

Our analysis will relay on the quasihomogeneity of the invariants  $z^i(x)$ in  the coordinates of $q(x)\in \lop \bneg$ and their derivatives. This property is summarized in the following corollary
\begin{cor}\label{lin inv poly}
If we assign degree $2 J+2l+2$ to  $\partial_x^l(q_i^J(x))$  then $z^i(x)$ will be quasihomogenous of degree $2\eta_i+2$. Furthermore, each invariant $z^i(x)$ depends linearly only on $q_i^{\eta_i}(x)$ and $\partial_x q_{i}^{\eta_i-1}(x)$. In particular, $z^i(x)$ with $i<n$ does not  depend on $\partial_x^l q^{\eta_r}_r(x)$ for any value $l$.
\end{cor}

Let us fix the following notations for the leading terms of the Drinfeld-Sokolov bihamiltonian structure on $Q$
\begin{eqnarray} \{z^i(x),z^j(y)\}_1^Q&=&
\sum_{k=-1}^\infty \epsilon^k \{z^i(x),z^j(y)\}^{[k]}_1 \\\nonumber
 \{z^i(x),z^j(y)\}_2^Q&=&
\sum_{k=-1}^\infty \epsilon^k \{z^i(x),z^j(y)\}^{[k]}_2.
 \end{eqnarray}
 where
\begin{eqnarray}\label{reducedPB notations}
  \{z^i(x),z^j(y)\}^{[-1]}_1 &=& F^{ij}_1(z(x))\delta(x-y) \\\nonumber
  \{z^i(x),z^j(y)\}^{[0]}_1 &=& g^{ij}_{1}(z(x)) \delta' (x-y)+ \Gamma_{1k}^{ij}(z(x)) z_x^k \delta (x-y)\\\nonumber
   \{z^i(x),z^j(y)\}^{[-1]}_2 &=& F^{ij}_2(z(x))\delta(x-y) \\\nonumber
  \{z^i(x),z^j(y)\}^{[0]}_2 &=& g^{ij}_{2}(z(x)) \delta' (x-y)+ \Gamma_{2k}^{ij}(z(x)) z_x^k \delta (x-y)
\end{eqnarray}

\subsection{The nondegeneracy  condition}

In this section we find the antidiagonal entries of the matrix $g^{ij}_1$ with respect to the  exponents of $\g$, i.e the entry $g^{ij}_1$ with $\eta_i+\eta_j=\kappa+1$. Our goal is to prove  this matrix is nondegenerate. 

Let $\Xi_I^i$ denote the value $\bil {X_I^i}{X_I^i}$ and we set \[[a,X_I^i]=\sum_j \Delta_I^{ij} X_{I-\eta_r}^j.\] By definition,  for a functional $\f$ on $\g$
\beq \delta \f(x)=\sum_i\sum_{I=0}^{\eta_i} {1\over \Xi_I^i}{\delta\f\over\delta q_i^I(x)} X_I^i\eeq
and  the Poisson brackets of two functionals $\I$ and $\f$ on $\g$ reads
\beq \{\I,\f\}_1=\bil {\delta \I(x)}{[a,\delta \f(x)]}=\sum_i\sum_{I=0}^{\eta_i}\sum_j {\Delta_I^{ij}\over \Xi_I^i}{\delta \I\over \delta q_j^{\kappa-I}(x)}{\delta\f\over\delta q_i^I(x)}.\eeq
Therefore, the Poisson brackets in coordinates have the form
\beq
\{q_j^{\kappa-I}(x),q_i^I(y)\}_1= {\Delta_I^{ij}\over \Xi_I^i}\delta(x-y).
\eeq

Recall that the  Poisson bracket $\{v(x),u(y)\}_1^Q$ of elements   $u,~v\in R$ is obtained  by  the Leibnitz rule which expands as
{\small \begin{eqnarray*}
& &\{ v(x),u(y)\}_1^Q= \sum_{i,I;j}\sum_{l,h}{\Delta_I^{ij}\over \Xi_I^i}{\partial v(x) \over \partial (q_j^{\kappa-I})^{(l)} }\partial_x^l\Big({\partial u(y) \over \partial (q_i^{I})^{(h)} }\partial_y^h(\delta(x-y))\Big)\\\nonumber
&=&  \sum_{i,I;j}\sum_{l,h,m,n} (-1)^h{h\choose m} {l\choose n}{\Delta_I^{ij}\over \Xi_I^i}{\partial v(x) \over \partial (q_j^{\kappa-I})^{(l)} }\Big({\partial u(x) \over \partial (q_i^{I})^{(h)} }\Big)^{m+n}\delta^{h+l-m-n}(x-y).
\end{eqnarray*}}
Here  we omitted the ranges of the indices  since no confusion can arise. Let $\mathbb{A}(v,u)$ denote the coefficient of $\delta'(x-y)$
\begin{equation}\label{formula for dispesionless limit}\mathbb{A}(v,u) =\sum_{i,I,J}\sum_{h,l}(-1)^h (l+h) {\Delta_I^{ij}\over \Xi_I^i}{\partial v(x) \over \partial (q_j^{\kappa-I})^{(l)} }\Big({\partial u(x) \over \partial (q_i^{I})^{(h)} }\Big)^{h+l-1}
\end{equation}
 Obviously, we  obtain the entry $g^{ij}_1$  from $\mathbb{A}(z^i,z^j)$.

 \begin{lem}  \label{homg of g1} If $\eta_i+\eta_j<\kappa+1$ then
 $\mathbb{A}(z^i,z^j)=0$. In particular, the matrix $g^{ij}_1$ is lower antidiagonal with respect to the exponents of $\g$ and the antidiagonal entries are constants.
 \end{lem}
\begin{proof}
We note that if  $v(x)$ and $u(x)$ are in $R$ and quasihomogenous of  degree $\theta$ and $\xi$, respectively, then $\mathbb{A}(v,u)$ will be quasihomogenous of degree \[ \theta+\xi-(2\kappa+2)-4.\]
The proof is complete by observing that the  generators $z^i(x)$ of the ring $R$  is quasihomogeneous  of degree $2\eta_i+2$.
\end{proof}
\begin{prop}\label{nondeg}
The matrix $g^{ij}_1$ is nondegenerate and its determinant is equal to the determinant of the matrix $A_{ij}$ defined in \eqref{Gold}.
\end{prop}
\begin{proof}
From the last lemma  we need only to consider the expression  $\mathbb{A}(z^n,z^m)$ with $\eta_n+\eta_m=\kappa+1$. Here
 \begin{equation}\mathbb{A}(z^n,z^m) =\sum_{i,I,J}\sum_{h,l}(-1)^h (l+h) {\Delta_I^{ij}\over \Xi_I^i}{\partial z^n(x) \over \partial (q_j^{\kappa-I})^{(l)} }\Big({\partial z^m(x) \over \partial (q_i^{I})^{(h)} }\Big)^{h+l-1}
\end{equation}
where  $z^m$ and $z^n$ are quasihomogenous of degree $2\eta_m+2$ and $2\kappa-2\eta_m+4$, respectively. The expression ${\partial z^m(x) \over \partial (q_i^{I})^{(h)} }$ gives the constrains
 \begin{eqnarray}
 2I+2\leq 2\eta_m+2 \\\nonumber
 2\kappa-2I+2 \leq 2\kappa-2 \eta_m+4
 \end{eqnarray}
which implies \[ \eta_m-1\leq I\leq \eta_m\]
Therefore the only possible values for the index $I$ in the expression of $\mathbb{A}(z^n,z^m)$ that make sense are $\eta_m$ and $\eta_m-1$. Consider  the partial summation of $ \mathbb{A}(z^n,z^m)$ when  $I=\eta_m$. The degree of $z^m$ yields $h=0$ and that $z^m$ depends linearly on $q_i^{\eta_m}$.  But then equation \eqref{leading terms1} implies $i$ is fixed and equals to $m$. A similar argument on $z^n(x)$ we find that the indices $l$ and $j$ are fixed and equal to $1$ and $n$, respectively. But then the  partial summation when $I=\eta_m$ gives the value
 \[{\Delta_{\eta_m}^{mn}\over \Xi_{\eta_m}^m}{\partial z^n(x) \over \partial (q_n^{\kappa-\eta_m})^{(1)} }{\partial z^m(x) \over \partial (q_m^{\eta_m})^{(0)}} =-{\Delta_{\eta_m}^{mn}\over \Xi_{\eta_m}^m}.\]
We now turn to  the partial summation of $ \mathbb{A}(z^n,z^m)$ when  $I=\eta_m-1$.  The  possible values for $h$ are $1$ and $0$. When $h=0$  we get zero since $l$ and $h$ can only be zero. When $h=1$ we get, similar to the above calculation,  the value
 \[
(-1)  {\Delta_{\eta_m-1}^{mn}\over \Xi_I^i}{\partial z^n(x) \over \partial (q_n^{\kappa-\eta_m})^{(0)} }{\partial z^m(x) \over \partial (q_m^{\eta_m-1})^{(1)} }={\Delta_{\eta_m-1}^{mn}\over \Xi_{\eta_m-1}^m}.
\]  Hence we end with the expression \begin{eqnarray*}
\mathbb{A}(z^n,z^m)&=& {\Delta_{\eta_m-1}^{mn}\over \Xi_{\eta_m-1}^m}-{\Delta_{\eta_m}^{mn}\over \Xi_{\eta_m}^m}\\\nonumber
&=&{\bil {[a,X_{\eta_n}^n]}{X_{\eta_m-1}^m}\over 2\eta_m}+ {\bil {[a,X_{\eta_m}^m]}{X_{\eta_n-1}^n}\over 2\eta_n}=A_{mn}
\end{eqnarray*}
where we derive the last equality in proposition \ref{Gold}. Hence the    determinate of $g^{ij}_1$  equals to the determinant of $A_{mn}$ which  is nondegenerate.
\end{proof}

\subsection{Differential relation}

We want to observe a differential relation between the first and the second Poisson brackets. This relation  is a consequence of the   fact that $z^r(x)$  is the only  generator of the ring $R$ which depends explicitly on $q_r^{\kappa}(x)$ and  this dependence is linear.

\begin{prop}\label{diff relaton} The entries of matrices of the reduced bihamiltonian structure on $Q$ satisfy the relations
\begin{eqnarray}
\partial_{z^r} F^{ij}_2 &=& F^{ij}_1\\\nonumber
\partial_{z^r} g^{ij}_2 &=& g^{ij}_1
\end{eqnarray}
\end{prop}
\begin{proof}
The fact that we calculate the reduced Poisson structure  by using Leibnitz rule and $z^r(x)$  depends on $q_r^{\kappa}(x)$  linearly, means that the invariant  $z^r(x)$ will appear on the reduced Poisson bracket $  \{z^i(x),z^j(y)\}_2^Q$ only as a result of the following ``brackets''
\begin{equation}
[q_j^{\kappa-I}(x),q_i^I(y)]:=q_r^{\kappa}(x) {\Delta^{ij}_I\over \Xi_I^i} \delta(x-y)
\end{equation}
which are the terms of the second Poisson bracket on $\lop \g$ depending explicitly on $q_r^{\kappa}(x)$.
 We expand the ``brackets'' $[z^i(x),z^j(y)]$ by imposing the Leibnitz rule. We find the coefficient of $\delta(x-y)$ and $\delta'(x-y)$  are, respectively,
\begin{eqnarray}
\mathbb{B}&=& \sum_{i,I,J}\sum_{h,l}(-1)^h {\Delta_I^{ij}\over \Xi_I^i}q_r^{\kappa}(x){\partial z^i(x) \over \partial (q_j^{\kappa-I})^{(l)} }\Big({\partial z^j(x) \over \partial (q_i^{I})^{(h)} }\Big)^{h+l}\\\nonumber
\mathbb{D}&=&\sum_{i,I,J}\sum_{h,l}(-1)^h (l+h) {\Delta_I^{ij}\over \Xi_I^i}q_r^{\kappa}(x){\partial z^i(x) \over \partial (q_j^{\kappa-I})^{(l)} }\Big({\partial z^j(x) \over \partial (q_i^{I})^{(h)} }\Big)^{h+l-1}
\end{eqnarray}
Obviously,  We have $\partial_{z^r} F^{ij}_2$ from $\partial_{q_r^{\kappa}} \mathbb{B}$   and  $\partial_{z^r} g^{ij}_2$ from  $\partial_{q_r^{\kappa}} \mathbb{D}$. But  we see that $\partial_{q_r^{\kappa}} \mathbb{D}$ is just the coefficient $\mathbb{A}(z^i,z^j)$ of $\delta'(x-y)$ of $\{z^i(x),z^j(y)\}_1^Q$. This prove  that \[\partial_{z^r} g^{ij}_2 = g^{ij}_1.\] A similar argument show that
\[\partial_{z^r} F^{ij}_2 = F^{ij}_1.\]
\end{proof}

\section{Some results from Dirac reduction}

We recall  that the Poisson bracket $\{.,.\}_2^Q$ can be obtained by performing the Dirac reduction of $\{.,.\}_2$ on $Q$. We derive from this some  facts  concerning   the dispersionless limit of the bihamiltonian structure on $Q$.  Let $\gdim$ denote   the dimension of $\g$.

Let  $\xi_I, ~I=1,\ldots,\gdim$  be a total order of  the basis ${X^i_I}$ such that
\begin{enumerate}
\item The first $ r $ are   \beq X^1_{-\eta_1}<X^2_{-\eta_2}<\ldots<X^r_{-\eta_r}\eeq
\item The matrix
\beq
\bil {\xi_I}{\xi_J},~ I,J=1,\ldots,\gdim
\eeq
is antidiagonal.
\end{enumerate}
Let $\xi_I^*$ denote  the dual basis of $\xi_I$ under $\bil . .$. Note that  if $\xi_I \in \g_{\mu}$ then $\xi_I^* \in \g_{-\mu}$.

We extend the coordinates on $Q$ to all $\lop \g$ by setting
\begin{equation}
\q^I(b(x)):=\bil{b(x)-e}{\xi^*_I},~ I=1,\ldots, \gdim.
\end{equation}
Let us fix  the following notations for the structure constants  and the bilinear form on $\g$
\begin{equation}
[\xi_I^*,\xi_J^*]:=\sum c^{IJ}_K \xi_K^*,~~~\widetilde{g}^{IJ}=\bil{\xi_I^*}{\xi_J^*}.
\end{equation}
Now consider the  following matrix differential operator
\begin{equation}\label{bih operator}
\mathbb{F}^{IJ}=\epsilon \widetilde{g}^{IJ}\partial_x+ \widetilde{F}^{IJ}.
\end{equation}
Here \[\widetilde F^{IJ}=\sum_{K}\big( c^{IJ}_K \q^K(x)\big).\]
 Then the Poisson brackets of  $P_2$  will have the form
\begin{equation}
\{\q^{I}(x),\q^J(y)\}_2=\mathbb{F}^{IJ}{1\over \epsilon}\delta(x-y).
\end{equation}

\begin{prop} \cite{BalFeh}\label{dirac red}
The second Poisson bracket $\{.,.\}_2^Q$ can be obtained by performing  Dirac reduction of $\{.,.\}_2$ on $Q$.
\end{prop}
A consequence of this proposition is the following
\begin{prop} \cite{BalFeh}\label{walg}
\begin{eqnarray}\label{leading terms}
\{z^1(x),z^1(y)\}_2&=& \eps \delta^{'''}(x-y) +2 z^1(x) \delta'(x-y)+ z^1_x\delta(x-y) \\\nonumber
\{z^1(x),z^i(y)\}_2 &=& (\eta_i+1) z^i(x) \delta'(x-y)+ \eta_i z^i_x \delta(x-y).
\end{eqnarray}
\end{prop}

\begin{rem}
The bihamiltonian reduction is a method introduced  in \cite{CMP} to reduced a bihamiltonian structure to a certain submanifold. We can use it  to obtain a bihamiltonian structure from \eqref{bih:stru on g} associated to the principal nilpotent element $e$ \cite{CP}. The resulting bihamiltonian structure is defined on $Q$. We generalize the bihamiltonian reduction in \cite{mypaper} by imposing  some conditions. The result is a bihamiltonian structure associated to any nilpotent element in a simple Lie algebra. This generalization also simplifies the bihamiltonian reduction given in \cite{CP}. The Drinfeld-Sokolov reduction is also generalized to any nilpotent element in simple Lie algebra \cite{fehercomp}. A similar result to proposition \ref{dirac red} for generalized Drinfeld-Sokolov reduction was obtained in  \cite{BalFeh}. We used it in \cite{mypaper2} to  prove that  the generalized  Drinfeld-Sokolov reduction and the generalized bihamiltonian reduction for any  nilpotent element are the same. This in turn complete the comparison between the two reductions  began by the work of Pedroni and Casati \cite{CP}. In \cite{mypaper} we also obtained proposition \ref{walg} by performing the generalized bihamiltonian reduction.
\end{rem}

For the rest of this section we   consider  three types of indices which have different ranges; capital letters $I,J,K,...=1,..,\gdim$,
small letters $i,j,k,...=1,....,r$  and Greek letters
$\alpha,\beta,\delta,...=r+1,...,\gdim$.  Recall that the space $Q$ is defined by  $\q^\alpha=0$.

We note that the matrix $\widetilde{F}^{IJ}$ define  the finite Lie-Poisson structure on $\g$.  It is well known that the  symplectic subspaces of this structure are the  orbit spaces of $\g$ under the adjoint group action and we have  $r$ global Casimirs \cite{MRbook}. Since the Slodowy slice $Q'=e+\g^f$ is transversal to the orbit of $e$, the minor matrix $\widetilde{F}^{\alpha \beta}$ is nondegenerate. Let $\widetilde{F}_{\alpha \beta}$ denote its inverse.

\begin{prop}\label{dirac-to-bihami}\cite{mypaper}
The Dirac formulas for the leading terms of $\{.,.\}_2^Q$ is given by
\begin{equation}\label{finite Poiss brac}
F_2^{ij}=(\widetilde{F}^{ij}-\widetilde{F}^{i\beta} \widetilde{F}_{\beta\alpha} \widetilde{F}^{\alpha j}
)\end{equation}
\begin{equation}\label{finite Poiss brac2}
g^{ij}_2= \widetilde{g}^{ij}-\widetilde{g}^{i\beta}\widetilde{ F}_{\beta\alpha}\widetilde{F}^{\alpha
j}+\widetilde{F}^{i \beta}\widetilde{F}_{\beta \alpha} \widetilde{g}^{\alpha \varphi} \widetilde{F}_{\varphi
\gamma} \widetilde{F}^{\gamma j}-\widetilde{F}^{i\beta} \widetilde{F}_{\beta \alpha} \widetilde{g}^{\alpha j}.
\end{equation}
\end{prop}

Now we are able to prove the following
\begin{prop}\label{the pencil}
The Drinfeld-Sokolov  bihamiltonian structure on $Q$ admits a dispersionless limit. The corresponding  bihamiltonian structure of hydrodynamic type  gives a flat pencil of metrics on the Slodowy slice $Q'$.
\end{prop}
\begin{proof}
We note that  \eqref{finite Poiss brac} is the formula of the Dirac reduction of the Lie-Poisson brackets of $\g$ to the finite space $Q'$. The fact that Slodowy slice is transversal to the orbit space of the nilpotent element and this orbit has dimension $\gdim-r$ yield $F^{ij}_2$ is trivial. From proposition \ref{walg}   it follows that $g^{ij}_2$ is not trivial. This prove that the brackets $\{.,.\}^Q_2$ admits a dispersionless limit. From propositions \ref{nondeg} and  \ref{diff relaton} it follows that $\{.,.\}_1^Q$ admits a dispersionless limit and the  matrix $g_2^{ij}$ is nondegenerate. Therefore, the two matrices $g_1^{ij}$ and $g_2^{ij}$ define a flat pencil of metrics on $Q'$.
\end{proof}

Now we want to study the quasihomogeneity  of the entries of the matrix $g^{ij}_2$.  We assign the  degree $\mu_I+2$ to $\q^I(x)$ if ${\xi^*_I}\in \g_{\mu_I}$. These degrees agree with those given in corollary \ref{lin inv poly}. We observe that  degree  $\q^{\gdim-I+1}$ equal to $-\mu_I+2$ from our order of the basis, and  an entry  $\widetilde F^{IJ}$ is quasihomogenous of degree  $\mu_I+\mu_J+2$ since $[\g_{\mu_I},\g_{\mu_J}]\subset \g_{\mu_I+\mu_J}$, .

The following proposition proved in \cite{DamSab}
\begin{prop}\label{mat poly}
The matrix $\widetilde{F}_{\beta \alpha} $  restricted to $Q$ is polynomial and the  entry $\widetilde{F}_{\beta \alpha} $ is quasihomogenous of degree $-\mu_\beta-\mu_\alpha-2$
\end{prop}

\begin{prop}\label{homg of g2}
The entry $g^{ij}_2$ is quasihomogenous of degree $2\eta_i+2\eta_j$
\end{prop}
\begin{proof}
We will derive the quasihomogeneity from the expression \eqref{finite Poiss brac2}. We know that the matrix $\widetilde g^{IJ}$ is constant antidiagonal, i.e  $g^{IJ}=C^I \delta^I_{\gdim-J+1}$ where $C^I$ are nonzero constants. In particular $g^{ij}=0$. Now  for a fixed $i$ we have
\[\widetilde{g}^{i\beta}\widetilde{ F}_{\beta\alpha}\widetilde{F}^{\alpha
j}= C^i \widetilde{ F}_{\gdim -i+1,\alpha}\widetilde{F}^{\alpha
j}.\]
 But then the left hand sight is quasihomogenous of degree \[\mu_j+\mu_\alpha+2 - \mu_\alpha-(-\mu_i)-2=\mu_j+\mu_i=2\eta_i+2\eta_j.\]
A similar argument show that  $ \widetilde{F}^{i\beta} \widetilde{F}_{\beta \alpha} \widetilde{g}^{\alpha j}$ is quasihomogeneous of degree $2\eta_i+2\eta_j$. Let us consider
\[\widetilde{F}^{i \beta}\widetilde{F}_{\beta \alpha} \widetilde{g}^{\alpha \varphi} \widetilde{F}_{\varphi
\gamma} \widetilde{F}^{\gamma j}=\sum_\alpha C^{\alpha} \widetilde{F}^{i \beta}\widetilde{F}_{\beta \alpha} \widetilde{F}_{ \gdim-\alpha+1,
\gamma} \widetilde{F}^{\gamma j}. \]
Then any term in this summation will have the degree
\[\mu_i+\mu_\beta+2-\mu_\beta-\mu_\alpha-2-\mu_{\gdim-\alpha+1}-\mu_\gamma-2+\mu_\gamma+\mu_j+2=2 \eta_i+2\eta_j\]
This complete the proof.
\end{proof}

\section{Polynomial Frobenius manifold}

Let us consider the finite dimension manifold $Q'$ defined by the coordinates $z^1,...,z^n$. We will obtain a natural polynomial Frobenius structure on  $Q'$.

 The proof of the following proposition depends only on the quasihomogeneity of the matrix $g^{ij}_1$.
\begin{prop}\cite{DCG}
There exist    quasihomogenous polynomials coordinates of degree $d_i$ in the form
\[ t^i=z^i+T^i(z^1,...,z^{i-1})\] such that  the matrix $g_1^{ij}(t)$   is  constant antidiagonal.
\end{prop}

For the remainder of this section, we fix  a  coordinates $(t^1,...,t^n)$ satisfying the proposition above. The following proposition emphasis that  under this change of coordinates some entries of the matrix $g_2^{ij}$ remain invariant.
\begin{prop}
The second metric $g^{ij}_2(t)$ and its Levi-Civita connection have the following entries
\begin{equation}
g^{1,n}_2(t)= (\eta_i+1) t^i, ~ \Gamma^{1j}_{2 k}(t)=\eta_j \delta^j_k
\end{equation}
\end{prop}
\begin{proof}
We know from proposition \ref{walg} that in the coordinates $z^i$ the matrix $g^{ij}_2(z)$ and its Levi-Civita connection  have the following entries
\begin{equation}
g^{1,n}_2(z)= (\eta_i+1) z^i, ~ \Gamma^{1j}_{2 k}(z)=\eta_j \delta^j_k
\end{equation}
Let $E'$ denote the Euler vector field give by
\beq
E'=\sum_i (\eta_i+1) z^i { \partial_{z^i}}.
\eeq
Then from the quasihomogeneity of $t^i$ we have $E'(t^i)=(\eta_i+1) t^i$. The formula for change of coordinates and the fact that $t^1=z^1$  give the following
\beq
g^{1j}(t)={\partial_{z^a} t^1 } {\partial_{z^b} t^j}~ g_2^{a b}(z)= E'(t^j)=(\eta_j+1) t^j.
\eeq
For the contravariant Levi-Civita connection the change of coordinates has the following formula
\beq
\Gamma^{ij}_{2k}(t) d t^k=\Big({\partial_{z^a} t^i } {\partial_{z^c}\partial_{z^b} t^j} g_2^{a b}(z)+ {\partial_{z^a} t^i} {\partial_{z^b} t^j  } \Gamma^{a b}_{2c}(z)\Big) d z^c.
\eeq
But then we get
\begin{eqnarray}
\Gamma^{1j}_{2k} d t^k&=&\Big( E' ({\partial_{z^c} t^j  })+  {\partial_{z^b} t^j  } \Gamma^{1 b}_{2c}\Big) d z^c\\\nonumber
&=& \Big( (\eta_j-\eta_c){\partial_{z^c} t^j  }+  \eta_c {\partial_{z^c} t^j  } \Big) d z^c=\eta_j {\partial_{z^c} t^j } d z^c =\eta_j d t^j
\end{eqnarray}

\end{proof}

We arrive to our basic result
\begin{thm}\label{my thm}
The flat pencil of metrics on the Slodowy slice $Q'$ obtained from the dispersionless limit of  Drinfeld-Sokolov bihamiltonian structure on $Q$ (see proposition \ref{the pencil}) is regular quasihomogenous of degree $\kappa-1\over \kappa+1$.
\end{thm}
\begin{proof}
In the notations of definition \ref{def reg} we take $\tau ={1\over \kappa+1}t^1$ then
\begin{eqnarray}
 E&=& g^{ij}_2 {\partial_{t^j} \tau  }~{\partial_{t^i}  }={1\over \kappa+1} \sum_{i} (\eta_i+1) t^i{\partial_{t^i} },\\\nonumber
 e &=&  g^{ij}_1 {\partial_{t^j} \tau }~{\partial_{t^i}  }={\partial_{t^r} }.
  \end{eqnarray}
We see immediately  that   \[ [e,E]=e\]
The identity \begin{equation} \Lie_e (~,~)_2 =
(~,~)_1 \end{equation} follows from and the fact that ${\partial_{t^r} }={\partial_{ z^r}}$ and proposition \ref{diff relaton}. The fact that
\begin{equation}
\Lie_e(~,~)_1
=0.
\end{equation}
 is a consequence  from the quasihomogeneity of the matrix $g_1^{ij}$ (see lemma \ref{homg of g1}). We also obtain from proposition \ref{homg of g2}
\begin{equation} \Lie_E (~,~)_2 =(d-1) (~,~)_2
\end{equation}
since
\begin{equation}
\Lie_E(~,~)_2(dt^i,dt^j)= E(g^{ij}_2)-{\eta_i+1\over \kappa+1}g^{ij}_2-{\eta_j+1\over \kappa+1} g^{ij}_2={-2\over \kappa+1} g^{ij}_2.
\end{equation}
The (1,1)-tensor
\begin{equation}
  R_i^j = {d-1\over 2}\delta_i^j + {\nabla_1}_iE^j = {\eta_i \over \kappa+1} \delta_i^j.
\end{equation}
Hence it is nondegenerate. This complete the proof.
\end{proof}

Now we are ready to prove theorem \ref{main thm}.

\begin{proof}{[Theorem \ref{main thm}]}
It follows from  theorem \ref{my thm} and \ref{dub flat pencil}  that $Q'$  has a Frobenius structure of degree $\kappa-1\over \kappa+1$ from the dispersionless limit of Drinfeld-Sokolov bihamiltonian structure. This Frobenius structure is polynomial since in the coordinates  $t^i$ the potential $\mathbb{F}$ is constructed from   equations \eqref{frob eqs} and we know from proposition \ref{mat poly}  that the matrix $g^{ij}_2$ is polynomial.
\end{proof}

\subsection{Conclusions and remarks}
The results of the present  work can be generalized  to  some class of distinguished  nilpotent elements in simple Lie algebras. In particular, we notice that  the existence  of  opposite Cartan subalgebras is  the main reason behind the examples of algebraic Frobenius manifolds constructed in \cite{mypaper} which are associated to distinguished nilpotent elements in the Lie algebra of type $F_4$.  In \cite{mypaper}  we discussed how these examples support Dubrovin conjecture. Our goal is  to develop  a method  to uniform the construction of all algebraic Frobenius manifolds that could be obtained from  distinguished nilpotent elements in simple Lie algebras by performing the generalized  Drinfeld-Sokolov reduction. Similar  treatment  of the present work  for algebraic Frobenius manifolds that could be obtained from   subregular nilpotent elements in simple Lie algebras  is now under preparation.\\

\noindent{\bf Acknowledgments.}

This work is partially supported by the European Science Foundation Programme ``Methods of Integrable Systems, Geometry,
Applied Mathematics" (MISGAM).

\end{document}